\documentclass[a4paper,11pt]{article}
\usepackage{amsmath,amssymb,enumerate,color}
\usepackage{amsthm,color}
\newcommand{\linktojournal}[1]{\relax}

\usepackage{hyperref}
\hypersetup{
setpagesize=false,
 bookmarksnumbered=true,%
 bookmarksopen=true,%
 colorlinks=true,%
 linkcolor=blue,
 citecolor=red
}

\newcommand{\R}{\mathbb{R}}

\newcommand{\Z}{\mathbb{Z}}
\newcommand{\N}{\mathbb{N}}
\newcommand{\ep}{\varepsilon}
\newcommand{\pa}{\partial}

\newcommand{\lr}[1]{{}\langle{}#1{}\rangle{}}

\newtheorem{theorem}{Theorem}[section]
\newtheorem{lemma}[theorem]{Lemma}
\newtheorem{proposition}[theorem]{Proposition}

\theoremstyle{remark}
\newtheorem{remark}{Remark}[section]
\theoremstyle{definition}

\newtheorem{definition}{Definition}[section]

\numberwithin{equation}{section}

\makeatletter
\def\@cite#1#2{[{{\bfseries #1}\if@tempswa , #2\fi}]}
\newcounter{hours}\newcounter{minutes}
\renewcommand*{\thehours}{\two@digits\c@hours}
\renewcommand*{\theminutes}{\two@digits\c@minutes}
\makeatother                  
\begin{document}
\begin{center}
\Large{{\bf
Higher order asymptotic expansion of solutions 
to \\abstract linear hyperbolic equations
}}
\end{center}

\vspace{5pt}

\begin{center}
Motohiro Sobajima%
\footnote{
Department of Mathematics, 
Faculty of Science and Technology, Tokyo University of Science,  
2641 Yamazaki, Noda-shi, Chiba, 278-8510, Japan,  
E-mail:\ {\tt msobajima1984@gmail.com}}
\end{center}

\newenvironment{summary}{\vspace{.5\baselineskip}\begin{list}{}{%
     \setlength{\baselineskip}{0.85\baselineskip}
     \setlength{\topsep}{0pt}
     \setlength{\leftmargin}{12mm}
     \setlength{\rightmargin}{12mm}
     \setlength{\listparindent}{0mm}
     \setlength{\itemindent}{\listparindent}
     \setlength{\parsep}{0pt}
     \item\relax}}{\end{list}\vspace{.5\baselineskip}}
\begin{summary}
{\footnotesize {\bf Abstract.}
The paper concerned with higher order asymptotic expansion 
of solutions to the Cauchy problem of 
abstract hyperbolic equations of the form 
$u''+Au+u'=0$ in a Hilbert space, where $A$ is a nonnegative selfadjoint operator. 
The result says that 
by assuming the regularity of initial data, 
asymptotic profiles (of arbitrary order) 
are explicitly written by using the semigroup $e^{-tA}$ generated by 
$-A$. To prove this, a kind of maximal regularity for $e^{-tA}$ 
is used. 
}
\end{summary}



\section{Introduction}
Let $H$ be a Hilbert space over $\R$ with the inner product $(\cdot,\cdot)$ and 
the norm $\|\cdot\|$. 
In this paper we consider higher order asymptotic expansion 
of solutions to abstract hyperbolic equations 
of the form 
\begin{align}\label{ADW}
\begin{cases}u''(t)+Au(t)+u'(t)=0, 
& t>0, 
\\
(u,u')(0)=(u_0,u_1),
\end{cases}
\end{align}
where $u:[0,t)\to H$ is an unknown function and 
$A$ is a nonnegative selfadjoint operator in $H$. 
We denote $D(A)$ as a domain of $A$.  
The pair $(u_0,u_1)\in D(A^{1/2})\times H$ is given. 

The problem \eqref{ADW} is motivated as the generalization of 
the damped wave equation 
\begin{equation}\label{dw}
\begin{cases}
\pa_t^2u(x,t)-\Delta u(x,t)+\pa_tu(x,t)=0, 
& (x,t)\in \R^N\times (0,\infty), 
\\
(u,u')(x,0)=(u_0(x),u_1(x)).
\end{cases}
\end{equation}
with $(u_0,u_1)\in H^1(\R^N)\times L^2(\R^N)$. 
The abstract framework as \eqref{ADW} 
is firstly introduced by Ikehata--Nishihara \cite{IkNi2003}.

The equation \eqref{dw} has been considered as a 
model of the phenomenon of heat conduction with 
finite propagation property (see Cattaneo \cite{Cattaneo} and Vernotte \cite{Vernotte}). 
Therefore the solution of \eqref{dw} 
is expected to have a similar profile of solutions 
to the heat equation
\begin{equation}\label{h}
\begin{cases}
\pa_tv(x,t)-\Delta u(x,t)=0, 
& (x,t)\in \R^N\times (0,\infty), 
\\
v(x,0)=v_0(x). 
\end{cases}
\end{equation}
Actually it is known that 
under the suitable assumption, 
the solution $u$ of \eqref{dw} behaves like the solution $v$ 
of \eqref{h} with $v_0=u_0+u_1$. 
This phenomenon is now so-called diffusion phenomenon 
for the damped wave equations. 
There are many investigation dealing with 
diffusion phenomena for \eqref{dw} 
(see e.g., Hsiao--Liu \cite{HsLi1992}, 
Nishihara \cite{Nishihara1997,Ni2003MathZ}, 
Karch \cite{Kar2000}, 
Yang--Milani \cite{YaMi2000} including 
some generalized problems like quasilinear systems). 
In the case of the damped wave equation in the exterior domain, 
Fourier analysis does not work well, and therefore, 
energy methods via integration by parts are often used 
(see e.g., Ikehata--Matsuyama \cite{IkMa2002}, 
Ikehata \cite{Ik2002} and Ikehata--Saeki \cite{IkSa2000}). 

Another example of the \eqref{ADW} is the damped beam equation 
\begin{equation}\label{db}
\begin{cases}
\pa_t^2u(x,t)+\pa_x^4u(x,t)-\alpha\pa_x^2 u(x,t)+\pa_tu(x,t)=\pa_x f\big(\pa_xu(x,t)\big), 
& (x,t)\in \R\times (0,\infty), 
\\
(u,u')(x,0)=(u_0(x),u_1(x)).
\end{cases}
\end{equation}
The global existence and asymptotic behavior of solutions to \eqref{db} 
are studied in Takeda--Yoshikawa \cite{TaYo2013, TaYo2012, TaYo2015} and
and Takeda \cite{Takeda2015}. 
In these papers, the asymptotic behavior of solutions of \eqref{db}
with $f\equiv 0$ 
plays a crucial role to consider global existence of nonlinear problem. 
Actually, in \cite{TaYo2012} the asymptotic behavior of solutions of nonlinear 
problem can be found as solutions of a linear problem.

For the further analysis, as an improvement of diffusion phenomena, 
the problem of higher order asymptotic expansion 
should be naturally considered. 
In 2001, Orive--Zuazua--Pazoto \cite{OrZuPa2001}, 
they considered the following general problem of \eqref{dw}:
\begin{equation}\label{dw-2}
\begin{cases}
\rho(x)\pa_t^2u(x,t)-{\rm div}\Big(a(x)\nabla u(x,t)\Big)+a_0\rho(x)\pa_tu(x,t)=0, 
& (x,t)\in \R^N\times (0,\infty), 
\\
(u,u')(x,0)=(u_0(x),u_1(x)),
\end{cases}
\end{equation}
where $\rho(x)$, $\rho(x)^{-1}$ and $a(x)=(a_{jk}(x))_{jk}$ are bounded 
and spatially periodic  
and the matrix $(a_{jk})_{jk}$ is uniformly and positively determined.  
In \cite{OrZuPa2001}, higher order asymptotic expansion 
of solutions to \eqref{dw-2} 
via the use of Bloch wave decomposition which is valid 
for spatial periodic coefficients. 
Later, Takeda \cite{Takeda2015} gave higher order asymptotic expansion 
by using the usual heat semigroup $e^{t\Delta}$ for \eqref{dw} 
via Fourier multiplier theory 
with asymptotic expansion of symbol of evolution operator with respect to 
the variable of Fourier spaces $\xi$.  
An asymptotic expansion for wave part can be found in a recent paper 
Michihisa \cite{Michihisa-arxiv}.

Instead of the various results in the previous works, 
in this paper the interest is how to systematically 
determine the asymptotic profile of arbitrary order. 
To discuss this problem we then consider 
an abstract hyperbolic equation \eqref{ADW} 
which is 
included three situations 
\eqref{dw} ($A=-\Delta$),  
\eqref{db} ($A=\frac{d^4}{dx^4}-\alpha\frac{d^2}{dx^2}$) 
and
\eqref{dw-2} ($Au=\rho^{-1}{\rm div}(a\nabla u)$).

Concerning the previous works of the abstract setting, 
Ikehata--Nishihara \cite{IkNi2003} introduced the problem \eqref{ADW} and 
observed the asymptotic behavior of solutions in the following way:
\[
\|u(t)-e^{-tA}(u_0+u_1)\|\leq C(1+t)^{-1}\Big(\log(e+t)\Big)^{1/2+\ep}, \quad (u_0,u_1)\in D(A)\times D(A^{1/2}).
\]
After that Chill--Haraux \cite{ChHa2003} discussed the same problem and 
succeeded in removing the logarithmic correction of the above inequality, 
which is conjectured in Ikehata--Nishihara \cite{IkNi2003}. 
Radu--Todorova--Yordanov \cite{RaToYo2011} 
studied also diffusion phenomena with respect to 
stronger norms $\|\cdot\|_{D(A^k)}$ (a similar analysis for a linear hyperbolic equation in Hilbert spaces
with time-dependent damping term $b(t)u'$ can be found in Yamazaki \cite{Yamazaki2006}).  
Radu--Todorova--Yordanov \cite{RaToYo2016} discussed 
a higher order approximation of solutions to $Bu''+Au+u'=0$ 
($B$ is bounded, selfadjoint and positively definite); 
however, their framework is only valid 
for semigroup in metric measure spaces $L^2(\Omega,\mu)$ 
with an extra maximal $L^p$-$L^q$ regularity.

The purpose of the present paper is to discuss higher order asymptotic expansion 
of solutions \eqref{ADW} by using the semigroup $e^{-tA}$ generated by $-A$. 
Moreover, the main topic is to give a way how to construct 
higher order asymptotic profiles of arbitrary order.

To begin with, we state existence of solutions to the problem \eqref{ADW}.
The following proposition follows from 
the standard theory for $C_0$-semigroup on product spaces 
(see e.g., Ikehata--Nishihara \cite{IkNi2003} and also Brezis \cite[Section X]{Brezis} when $A=-\Delta$). 
\begin{proposition}
Let $A$ be a nonnegative selfadjoint operator in $H$ with domain 
$D(A)$. Then the operator 
\[
\mathcal{A}
=\begin{pmatrix}
0 & -1\\
A &  1
\end{pmatrix}, 
\]
in $\mathcal{H}=D(A^{1/2})\times H$ with domain 
$D(\mathcal{A})=D(A)\times D(A^{1/2})$
is quasi-$m$-accretive in $\mathcal{H}$. 
Namely, $-\mathcal{A}$ generates a $C_0$-semigroup in $\mathcal{H}$. 

In particular, the solution $u$ of \eqref{ADW} uniquely exists in the following sense:
\begin{itemize}
\item[\bf (i)]
for $(u_0,u_1)\in D(A^{1/2})\times D(H)$, one has 
$(u(t),u'(t))=e^{-t\mathcal{A}}(u_0,u_1)$ which implies 
\[
u\in C^2([0,\infty);[D(A^{1/2})]^*)\cap C^1([0,\infty);H)\cap C([0,\infty);D(A^{1/2})).
\]
and the equation \eqref{ADW} is satisfied in $[D(A^{1/2})]^*$ (the dual space of $D(A^{1/2})$), that is, 
\[
\lr{u''(t),\varphi}_{[D(A^{1/2})]^*,D(A^{1/2})}
+
(A^{1/2}u(t),A^{1/2}\varphi)
+(u'(t),\varphi)=0, \quad \forall\varphi\in D(A^{1/2})
\] 
\item[\bf (ii)]
for $(u_0,u_1)\in D(A^{\frac{n+1}{2}})\times D(A^{\frac{n}{2}})\ (n\in\N)$, one has 
$(u(t),u'(t))=e^{-t\mathcal{A}}(u_0,u_1)$ which implies 
\[
u\in \bigcap_{k=0}^{n+1}C^{k}([0,\infty);D(A^{\frac{n-k}{2}}))
\]
\end{itemize}
\end{proposition}

Now we are in a position to give the result for 
higher order asymptotic expansion for 
abstract hyperbolic equation \eqref{ADW}, 
which is the main result of the present paper. 
\begin{theorem}\label{thm:main}
Assume that $(u_0,u_1)\in D(A^{m+1/2})\times D(A^{m+1/2})$ 
for some $m\in \N\cup \{0\}$. 
Let $u$ be a unique solution of \eqref{ADW}. 
Define $v_0=u_0+u_1$ and for $\ell\in \N\cup\{0\}$, 
\begin{align*}
\overline{u}_0(t)
&=e^{-tA}v_0, 
\\
\overline{u}_\ell(t)
&=
A^\ell\left(
\sum_{j=0}^\ell
\begin{pmatrix}
2\ell-1\\\ell+j-1
\end{pmatrix}\frac{(-tA)^j}{j!}e^{-tA}v_0
+
\sum_{k=0}^{\ell-1}
\begin{pmatrix}
2\ell-1\\ \ell+k
\end{pmatrix}\frac{(-tA)^k}{k!}e^{-tA}u_1
\right).
\end{align*}
Then there exists a positive constant $C$ such that for every $t\geq 0$, 
\begin{align*}
\left\|u(t)-\sum_{\ell=0}^m\overline{u}_\ell(t)\right\|
&\leq C(1+t)^{-m-1/2}.
\end{align*}
\end{theorem}
\begin{remark}\label{rem:0eigen}
Especially, if $0$ is an eigenvalue of $A$ and $\varphi_0$ is the corresponding eigenvector, then $e^{-tA}\varphi_0=\varphi_0$. Therefore there is no possibility 
to improve the upper bound \eqref{eq:contraction} 
in the general setting. 
Since the asymptotic profiles $\{\overline{u}_\ell\}$ satisfy
\begin{align}\label{eq:contraction}
\|\overline{u}_\ell(t)\|\leq C_\ell(1+t)^{-\ell},
\end{align}
the assertion in Theorem \ref{thm:main} can be understood as 
a higher order asymptotic expansion of solution $u$ of \eqref{ADW}.
Of course in the several situation such as $-\Delta$ in $\R^N$, 
we have $\|e^{-tA}(u_0+u_1)\|\leq C(1+t)^{-\alpha}$ 
for some $\alpha>0$. In this case we need some effort 
to determine especially the regularity of initial data. 
We will not touch the details of specialized cases. 
\end{remark}
\begin{remark}
If $m=0$, then Theorem \ref{thm:main} is weaker than 
those of Ikehata--Nishihara \cite{IkNi2003} and also 
Chill--Haraux \cite{ChHa2003}. 
Since the main topic of our result is 
to give a higher order asymptotic profiles, 
the optimality of decay rates is not precisely discussed.  

\end{remark}
\begin{remark}
In the case \eqref{dw}, 
we choose $H=L^2(\R^N)$ 
and $A=-\Delta$ endowed with domain $D(A)=H^2(\R^N)$.  
Takeda \cite{Takeda2015} obtained the same asymptotic expansion 
with a different expression
\begin{align*}
\overline{u}_\ell(t)&=\frac{1}{2}
\sum_{j=0}^\ell\alpha_{j,k}(-t)^j(-\Delta)^{j+\ell}e^{t\Delta}u_0
\\
&\quad+
\sum_{0\leq k_1+k_2\leq \ell}\alpha_{\ell-k_1-k_2,k_1}\beta_{k_2}(-t)^{\ell-k_1-k_2}(-\Delta)^{2\ell-k_1-k_2}e^{t\Delta}\left(\frac{1}{2}u_0+u_1\right),
\end{align*}
where 
\[
\alpha_{j,k}
=\frac{1}{j!k!}\left.\frac{d^k}{dr^k}\phi_j\right|_{r=0},
\quad
\phi_j(r)=\left(\frac{1}{2}+\sqrt{\frac{1}{4}-r}\right)^{-2j},\]
\[
\beta_{k}
=\frac{1}{k!}\left.\frac{d^k}{dr^k}\psi\right|_{r=0}, 
\quad 
\psi(r)=
\frac{1}{2}\left(\frac{1}{4}-r\right)^{-1/2}.
\]
This result is valid for the whole space case $\R^N$ 
and the effect of high-frequency part is clearly written. 
On the other hand, Theorem \ref{thm:main} asserts 
that the same asymptotic expansion is valid 
for every nonnegative selfadjoint operator in a Hilbert space. 
\end{remark}
\begin{remark}
In the case \eqref{db}, we choose 
$H=L^2(\R)$ and $A=\frac{d^4}{dx^4}-\alpha \frac{d^2}{dx^2}$ 
endowed with domain $D(A)=H^4(\R)$. 
According to Theorem \ref{thm:main} with $m=0$, 
we can see that 
\[
\|u(t)-e^{-tA}(u_0+u_1)\|\leq C(1+t)^{-\frac{1}{2}}. 
\]
On the other hand, in Takeda--Yoshikawa \cite{TaYo2012}, 
it is shown that the asymptotic behavior of the solution $u$ 
is given by $e^{-\alpha tA_0}(u_0+u_1)$, where $A_0=-\frac{d^2}{dx^2}$.
The observation in Theorem \ref{thm:main} does not have contradiction 
because of the following 
estimate
\[
\Big\|e^{-tA}(u_0+u_1)-e^{-\alpha tA_0}(u_0+u_1)\Big\|
\leq \frac{C}{t}\|tA_0e^{-\alpha tA_0}A_0(u_0+u_1)\|, 
\quad u_0+u_1\in H^2(\R). 
\]
However, if we choose the latter profile, 
then 
since the asymptotic expansion in Theorem \ref{thm:main} 
is written by the operator $A$ ($\neq A_0$), 
the their difference should be carefully analysed. 
\end{remark}
\begin{remark}
In the case \eqref{dw-2} (with bounded $\rho,\rho^{-1}$), we choose 
$H=L^2(\R)$ but an inner product different from the usual one: 
\[
(f,g)_{H}=\int_{\R^N}fg\rho\,dx.
\]
Then we use  $Au=\rho(x)^{-1}{\rm div}(a(x)\nabla u)$ endowed with domain $D(A)=H^2(\R^N)$. 
Of course we can consider the periodic setting $\mathbb{T}^N=(\R/\Z)^N$ by choosing Sobolev spaces of periodic functions. 
Combining the strategy of Bloch wave decomposition, 
we can also deduce a similar result in \cite{OrZuPa2001}. 
\end{remark}
Let us describe the strategy for 
the construction of asymptotic profiles.
It is well known that the solution $u$ of \eqref{ADW} has the energy decreasing property
\[
E(u,t)
+
2\int_0^t\|u'(s)\|^2\,ds=E(u,0)=\|u_1\|^2+\|A^{1/2}u_0\|^2,
\]
where the energy functional of the solution $u$ of \eqref{ADW} is defined as 
\[
E(u;t)=\|u'(t)\|^2+\|A^{1/2}u(t)\|^2.
\]
According to the experiences, we expect that 
the first order asymptotic profile of $u$ is $e^{-tA}(u_0+u_1)$.
On the other hand, 
if $w$ is the solution of \eqref{ADW}, then by a direct computation we have
\[
(1+t)\|w'(t)\|^2\leq (1+t)E(w,t)+\|w(t)\|^2\leq C.
\]
Combining this concept, we consider the following auxiliary problem 
\begin{equation}\label{intro.aux}
\begin{cases}
U_1''+AU_1+U_1'=Ae^{tA}(u_0+u_1), \quad t>0,
\\
(U_1,U_1')(0)=(0,-u_1).
\end{cases}
\end{equation}
After some computation, we have $u(t)=e^{-tA}(u_0+u_1)+U_1'(t)$, 
which is the first decomposition. 
It should be noticed that the function $U_1$ is completely the same as 
the function $Z$ in Ikehata--Nishihara \cite{IkNi2003}, 
however, the treatment of $U_1$ is different from \cite{IkNi2003} 
in which $Z$ is understood as the solution of $Z'+AZ=u'$, $Z(0)=0$. 
The idea of \eqref{intro.aux} can be understood as another use of 
the modified Morawetz method which is written in Ikehata--Matsuyama \cite{IkMa2002}. 

Then analysing the asymptotic profile of the solution $U_1$ 
and 
proceeding the similar argument as before, 
we successively obtain the sequence ${U_\ell}$ and their corresponding hyperbolic problems 
similar to  \eqref{intro.aux}. 
Finally, to obtain the desired decay property of error terms (such as $U_1'$), 
we use a kind of maximal regularity for 
the Cauchy problem of parabolic equation 
$V'+AV'=0$. This consideration suggests that the solution $u$ of \eqref{ADW} 
can be decomposed by 
\[
u(t)=\sum_{\ell=0}^m \frac{d^\ell}{dt^\ell}V_{\ell}(t)+\frac{d^{m+1}}{dt^{m+1}}U_{m+1}(t).
\]

The present paper is organized as follows. 
In Section 2, as a preliminary, we state and prove 
a kind of maximal regularity result  
the Cauchy problem of parabolic equation 
$V'+AV'=0$, which we will use later. 
In Section 3, we provide a proof of the simplest case $m=0$ 
in Theorem \ref{thm:main} 
to clarify how to show the diffusion phenomena. 
The first part of Section 4 provides 
a construction of 
a family of function which will be understood as higher order asymptotic profiles.
Finally, at the rest of Section 4 we show higher order asymptotic expansion 
of solutions to \eqref{ADW}
(Theorem \ref{thm:main} for $m\in \N$). 
\section{Preliminary result for the property of the semigroup $e^{-tA}$}

In this section, we give recall a kind of 
maximal regularity result for the abstract (selfadjoint) semigroup $e^{-tA}$. 
For the reader's convenience, we provide a short proof.   
\begin{lemma}\label{lem:max-reg}
If $f\in H$, then for every $n\in \N\cup\{0\}$ and $t\geq 0$, 
\begin{equation}\label{eq:lem1-1}
\frac{\|e^{-tA}f\|^2}{2}
+
\frac{2^{n}}{n!}\int_0^t s^{n}\|A^{\frac{n+1}{2}}e^{-sA}f\|^2\,ds=\frac{\|f\|^2}{2}.
\end{equation}
Moreover, if $f\in D(A^{n/2})$, then 
there exists a positive constant $C$ such that 
\[
\int_0^t (1+s)^{n}\|A^{\frac{n+1}{2}}e^{-sA}f\|^2\,ds\leq C(\|f\|^2+\|A^{n/2}f\|^2).
\]
\end{lemma}

\begin{proof}
If $f\in D(A^{n/2+1})$, then $w(t)=e^{-tA}f$ satisfies
\begin{align*}
\frac{d}{dt}\|A^{k/2}w(t)\|^2
&=-2(A^{k/2}w(t),A^{k/2}w'(t))
\\
&=-2(A^{k/2}w(t),A^{k/2+1}w(t))
\\
&=-2\|A^{\frac{k+1}{2}}w(t)\|^2.
\end{align*}
In view of the above computation, in particular, we have
\begin{align*}
\frac{d}{dt}\Big[t^{k+1}\|A^{\frac{k+1}{2}}w(t)\|^2\Big]
&=(k+1) t^{k}\|A^{\frac{k+1}{2}}w(t)\|^2-2t^{k+1}\|A^{\frac{k+2}{2}}w(t)\|^2
\end{align*}
and therefore 
\begin{align*}
\frac{d}{dt}\Big[
\sum_{k=0}^{n-1}\frac{(2t)^{k+1}}{(k+1)!}\|A^{\frac{k+1}{2}}w(t)\|^2\Big]
&=
2\sum_{k=0}^{n-1}\left(
\frac{(2t)^{k}}{k!}\|A^{\frac{k+1}{2}}w(t)\|^2
-
\frac{(2t)^{k+1}}{(k+1)!}\|A^{\frac{k+2}{2}}w(t)\|^2
\right)
\\
&=
\|A^{1/2}w(t)\|^2
-
\frac{(2t)^{n}}{n!}\|A^{\frac{n+1}{2}}w(t)\|^2.
\end{align*}
Integrating it over $[0,t]$, we obtain \eqref{eq:lem1-1}. 
\end{proof}

\section{The first asymptotics}
In this section, 
we will show the strategy how 
to find a asymptotic profile of solutions 
in the simplest situation $m=0$ in Theorem \ref{thm:main}. 

To begin with, we prove the following lemma.
\begin{lemma}\label{lem:0th-en}
Assume that $(u_0,u_1)\in \mathcal{H}$ 
and let $u$ be a unique solution of \eqref{ADW}. 
Then there exists a positive constant $C>0$ such that 
for every $t\geq 0$, 
\begin{align*}
(1+t)E(u,t)+\|u(t)\|^2
+
\int_0^t
\Big((1+s)\|u'(s)\|^2+\|A^{1/2}u(s)\|^2\Big)\,ds
\leq 
C\Big(E(u,0)+\|u_0\|^2\Big).
\end{align*}
\end{lemma}
\begin{proof}
By taking a suitable approximation of $(u_0,u_1)$, 
We may assume $(u_0,u_1)\in D(\mathcal{A})$ without loss of generality.  
Then  we see by a direct computation 
with the equation \eqref{ADW} that 
\begin{align*}
\frac{d}{dt}E(u;t)
&=
2(u',u'')+2(A^{1/2}u',A^{1/2}u)
\\
&=
2(u',u''+Au)
\\
&=
-2\|u'\|^2.
\end{align*}
This yields
\[
\|u'(t)\|^2+\|A^{1/2}u\|^2+2\int_0^t\|u'(s)\|^2\,ds
=\|u_1\|^2+\|A^{1/2}u_0\|^2.
\]
On the other hand, using the equation \eqref{ADW}, we also have the following two 
identities:
\begin{align*}
\frac{d}{dt}\Big[(3+t)E(u;t)\Big]
&=
\|u'\|^2+\|A^{1/2}u\|^2-2(3+t)\|u'\|^2,
\\
\frac{d}{dt}E_*(u;t)
&=2\|u'\|^2+2(u,u''+u')
\\
&=2\|u'\|^2-2\|A^{1/2}u\|^2.
\end{align*}
Summing up the above identities, we deduce
\begin{align}\label{eq:diff-sum}
\frac{d}{dt}\Big[(3+t)E(u;t)+E_*(u;t)\Big]
=-(3+2t)\|u'\|^2-\|A^{1/2}u\|^2.
\end{align}
Noting that the differentiated function in \eqref{eq:diff-sum} can be rewritten as
\begin{align*}
(3+t)E(u;t)+E_*(u;t)
=(1+t)E(u,t)+2\|A^{1/2}u\|^2+2\left\|u'+\frac{1}{2}u\right\|^2+\frac{1}{2}\|u\|^2,
\end{align*}
we obtain the desired inequality. 
\end{proof}

The next lemma enables us 
to divide the solution $u$ of \eqref{ADW} into 
two factors: one reflects the diffusion phenomenon 
and the other reflects an effect of hyperbolicity. 
This treatment is the difference between the one in 
Ikehata--Nishihara \cite{IkNi2003} and ours. 

\begin{lemma}\label{lem:order1}
Let $(u_0,u_1)\in D(A^{1/2})\times D(A^{1/2})$ 
and let $U_1$ be the unique solution of the problem
\begin{align}\label{ADW-order1}
\begin{cases}
U_1''+AU_1+U_1'=Ae^{tA}(u_0+u_1), \quad t>0,
\\
(U_1,U_1')(0)=(0,-u_1).
\end{cases}
\end{align}
Then the function $e^{-tA}(u_0+u_1)+U_1'$ coincides with the solution of \eqref{ADW}. 
\end{lemma}
\begin{proof}
Take $(u_{0\ep}, u_{1\ep})\in D(A^2)\times D(A^2)$
such that $u_{0\ep}\to u_0$ and $u_{1\ep}\to u_1$ in $D(A^{1/2})$ as $\ep \to 0$ 
and consider
\begin{align}\label{ADW-order1ep}
\begin{cases}
U_{1\ep}''+AU_{1\ep}+U_{1\ep}'=Ae^{tA}(u_{0\ep}+u_{1\ep}), \quad t>0,
\\
(U_{1\ep},U_{1\ep}')(0)=(0,-u_{1\ep}).
\end{cases}
\end{align}
Since $Ae^{tA}(u_{0\ep}+u_{1\ep})\in C([0,\infty);D(A))$,
this problem has a unique solution 
\[
U_{1\ep}\in 
C^2([0,\infty);D(A^{1/2}))\cap C^1([0,\infty);D(A))\cap C([0,\infty);D(A^{3/2})).
\]
Moreover, by using Duhamel's principle we can see that 
the solution $U_{1\ep}$ can be represented by 
\[
(U_{1\ep}(t),U_{1\ep}'(t))=e^{-t\mathcal{A}}(0,-u_{1\ep})+
\int_0^te^{-(t-s)\mathcal{A}}(0,Ae^{-sA}(u_{0\ep}+u_{1\ep}))\,ds
\]
which converges to the solution of \eqref{ADW-order1} in $\mathcal{H}$ as $\ep\to 0$:
\[
e^{-t\mathcal{A}}(0,-u_{1})+
\int_0^te^{-(t-s)\mathcal{A}}(0,Ae^{sA}(u_{0}+u_{1}))\,ds=(U_1(t),U_1'(t)).
\]
Setting $v_\ep(t)=e^{-tA}(u_{0\ep}+u_{1\ep})$ and $w_\ep(t)=v_\ep(t)+U_{1\ep}'(t)$, 
we have 
\[
w_\ep(0)=u_{0\ep}+u_{1\ep}-u_{1\ep}=u_{0\ep}
\]
and 
\begin{align*}
w_\ep'(t)&=v_\ep'(t)+U_{1\ep}''(t)
\\
&=v_\ep'(t)+Av_\ep(t)-AU_{1\ep}(t)-U_{1\ep}'(t)
\\
&=-AU_{1\ep}(t)-U_{1\ep}'(t). 
\end{align*}
This yields
\[
\lim_{t\to 0}w_\ep'(t)=-AU_{1\ep}(0)-U_{1\ep}'(0)=u_{1\ep}
\]
and also $w_\ep'\in C^1((0,\infty);\mathcal{H})$ with 
\begin{align*}
w_\ep''(t)
&=-AU_{1\ep}'(t)-(w_\ep(t)-v_\ep(t))'
\\
&=-A(w_\ep(t)-v_\ep(t))-w_\ep'(t)+v_\ep'(t)
\\
&=-Aw_\ep(t)-w_\ep'(t).
\end{align*}
This implies $(w_\ep(t),w_\ep'(t))=e^{-t\mathcal{A}}(u_{0,\ep},u_{1\ep})$ and therefore $(w_\ep(t),w_\ep'(t))\to (u(t),u'(t))$ in $\mathcal{H}$ as $\ep\to 0$.  
Consequently, 
noting that $v_\ep(t)\to e^{-tA}(u_0+u_1)$ in $D(A^{1/2})$ as $\ep\to 0$
we have 
\[
u(t)=\lim_{\ep\to 0}w_\ep(t)=e^{t\Delta}(u_0+u_1)+U_1'(t). 
\]
The proof is complete. 
\end{proof}

Observe that 
a decaying property of $e^{-tA}(u_0+u_1)$ 
cannot be expected as mentioned in Remark \ref{rem:0eigen}. 
In contrast, the other factor $U'$ decays like $(1+t)^{-1/2}$ 
under a suitable assumption. 
To prove the decay property 
we proceed a similar way in the proof of Lemma \ref{lem:0th-en}, 
however, we need Lemma \ref{lem:max-reg} which is 
a kind of maximal regularity result for $e^{-tA}$. 

\begin{lemma}\label{lem:en-U_1}
If $(u_0,u_1)\in D(A^{1/2})\times D(A^{1/2})$, then 
\[
(1+t)E(U_1;t)+\|U_1(t)\|^2\in L^\infty((0,\infty)), 
\quad 
(1+t)\|U_1'(t)\|^2+\|A^{1/2}U_1(t)\|^2\in L^1((0,\infty))
\]
\end{lemma}

\begin{proof}
By using the same approximation as in the proof of Lemma \ref{lem:order1}, 
we assume may assume 
$U_1\in C^2([0,\infty);D(A^{1/2}))\cap C^1([0,\infty);D(A))\cap C([0,\infty);D(A^{3/2}))$. 
Here 
we deduce an energy inequality 
for $U_1$. 
By the equation \eqref{ADW-order1}, we have
\begin{align*}
\frac{d}{dt}E_*(U_1;t)
&=2\|U_1'\|^2+2(U_1,-AU_1+Ae^{-tA}(u_0+u_1))
\\
&=2\|U_1'\|^2-2\|A^{1/2}U_1\|^2+2(U_1,Ae^{-tA}(u_0+u_1))
\\
&\leq 2\|U_1'\|^2-\|A^{1/2}U_1\|^2+\|A^{1/2}e^{-tA}(u_0+u_1)\|^2
\end{align*}
and 
\begin{align*}
\frac{d}{dt}\Big[(6+t)E(U_1;t)\Big]
&=\|U_1'\|^2+\|A^{1/2}U_1\|^2
+2(6+t)(U_1',-U_1'+Ae^{-tA}(u_0+u_1))
\\
&\leq\|U_1'\|^2+\|A^{1/2}U_1\|^2
-(6+t)\|U_1'\|^2
+(6+t)\|Ae^{-tA}(u_0+u_1)\|^2
\end{align*}
In view of Lemma \ref{lem:max-reg} with $f=u_0+u_1\in D(A^{1/2})$, 
we have
\begin{align*}
&\frac{d}{dt}\Big[(6+t)E(U_1;t)+2E_*(U_1;t)\Big]
+(1+t)\|U_1'\|^2+\|A^{1/2}U_1\|^2
\\
&\leq 
(6+t)\|Ae^{-tA}(u_0+u_1)\|^2+2\|A^{1/2}e^{-tA}(u_0+u_1)\|^2
\end{align*}
which is integrable with respect to $t\in(0,\infty)$. 
Integrating the above estimate over $[0,t]$, we obtain all desired estimates. 
\end{proof}

Combining Lemmas \ref{lem:0th-en}, \ref{lem:order1} and \ref{lem:en-U_1}, 
we obtain the following proposition which is the same as 
the assertion of Theorem \ref{thm:main} with $m=0$. 
\begin{proposition}\label{prop:1st-aym}
Assume $(u_0,u_1)\in D(A^{1/2})\times D(A^{1/2})$. Then 
there exists a positive constant $C>0$ such that for every $t\geq 0$, 
\begin{align*}
\|u(t)-e^{-tA}(u_0+u_1)\|\leq C(1+t)^{-\frac{1}{2}}.
\end{align*}
\end{proposition}

\section{Higher order asymptotic expansion}

In this section, we prove Theorem \ref{thm:main} with $m\in\N$ 
which gives higher order asymptotic expansion of
solutions to \eqref{ADW}. 
The philosophy of the proof
is essentially the same as the previous section. 

\subsection{Construction of family of functions for asymptotics}
To consider higher order asymptotic expansions, 
we define the following functions. 

\begin{definition}\label{def:V_m}
For $m\in\N$ and $(u_0,u_1)\in \mathcal{H}$, define $v_0=u_0+u_1$ and 
\begin{align*}
V_0^{(1)}(t)
&=
e^{-tA}v_0, 
\\
V_m^{(1)}(t)
&=
(-1)^{m}\sum_{j=1}^m
\begin{pmatrix}
m-1\\j-1
\end{pmatrix}\frac{(-tA)^j}{j!}e^{-tA}v_0,\quad m\in\N
\\
V_0^{(2)}(t)
&=0, 
\\
V_m^{(2)}(t)
&=
(-1)^{m}
\sum_{k=0}^{m-1}
\begin{pmatrix}
m-1\\k
\end{pmatrix}\frac{(-tA)^k}{k!}e^{-tA}u_1,\quad m\in\N
\end{align*}
and $V_m(t)=V_m^{(1)}(t)+V_m^{(2)}(t)$.
\end{definition}
The family $\{V_{m}\}$ will provide higher order asymptotic profiles;
this will be explained after stating Lemma \ref{lem:orderm} below.
 
The function $V_m$ can be successively found by using the following structure. 
\begin{lemma}\label{lem:asm:succ}
Assume $u_1\in D(A^{1/2})$. Then 
for $m\in \N\cup\{0\}$, $V_{m+1}$ is a unique solution of 
\[
\begin{cases}
V_{m+1}'+AV_{m+1}=-V_{m}', \quad t>0,
\\
V_{m+1}(0)=(-1)^{m+1}u_1.
\end{cases}
\]
\end{lemma}
\begin{proof}
We use the Duhamel principle. The solution of 
\[
\begin{cases}
w_{m+1}'+Aw_{m+1}=-V_{m}', \quad t>0,
\\
w_{m+1}(0)=(-1)^{m+1}u_1.
\end{cases}
\]
can be written by the formula
\begin{align*}
w_{m+1}(t)=e^{-tA}\Big[(-1)^{m+1}u_1\Big]
-\int_0^te^{-(t-s)A}V_{m}'(s)\,ds.
\end{align*}
For the case $m=0$, we see from the semigroup property that 
\begin{align*}
w_1(t)
&=e^{-tA}\Big[-u_1\Big]-\int_0^te^{-(t-s)A}V_{0}'(s)\,ds
\\
&=-e^{-tA}u_1-\int_0^te^{-(t-s)A}(-Ae^{sA}v_0)\,ds
\\
&=-e^{-tA}u_1+tAe^{-tA}v_0
\\
&=V_1^{(1)}(t)+V_1^{(2)}(t).
\end{align*}
This gives the assertion with $m=0$. For the case $m\in\N$, 
observe that 
\begin{align*}
&\frac{d}{dt}V_{m}^{(1)}(t)
\\
&=
(-1)^{m}
\left[\sum_{j=1}^{m}
\begin{pmatrix}
m-1\\j-1
\end{pmatrix}\frac{(-tA)^{j-1}}{(j-1)!}(-A)e^{-tA}v_0
+
\sum_{j=1}^{m}
\begin{pmatrix}
m-1\\j-1
\end{pmatrix}\frac{(-tA)^{j}}{j!}(-A)e^{-tA}v_0
\right]
\\
&=
(-1)^{m}
\left[\sum_{j=1}^{m}
\begin{pmatrix}
m-1\\j-1
\end{pmatrix}\frac{(-tA)^{j-1}}{(j-1)!}(-A)e^{-tA}v_0
+
\sum_{j=2}^{m+1}
\begin{pmatrix}
m-1\\j-2
\end{pmatrix}\frac{(-tA)^{j-1}}{(j-1)!}(-A)e^{-tA}v_0
\right]
\\
&=
(-1)^{m}
\sum_{j=1}^{m+1}
\begin{pmatrix}
m\\j-1
\end{pmatrix}\frac{(-tA)^{j-1}}{(j-1)!}(-A)e^{-tA}v_0.
\end{align*}
Therefore
\begin{align*}
\int_0^te^{-(t-s)A}\Big(\frac{d}{dt}V_{m}^{(1)}(s)\Big)\,ds
&=
\int_0^te^{-(t-s)A}
\left((-1)^{m}
\sum_{j=1}^{m+1}
\begin{pmatrix}
m\\j-1
\end{pmatrix}\frac{(-sA)^{j-1}}{(j-1)!}(-A)e^{-sA}v_0
\right)\,ds
\\
&=
(-1)^{m}
\sum_{j=1}^{m+1}
\begin{pmatrix}
m\\j-1
\end{pmatrix}
\left(
\int_0^t
\frac{s^{j-1}}{(j-1)!}\,ds
\right)(-A)^{j}e^{-tA}v_0
\\
&=
(-1)^{m}
\sum_{j=1}^{m+1}
\begin{pmatrix}
m\\j-1
\end{pmatrix}
\frac{(-tA)^{j}}{j!}e^{-tA}v_0
\\
&=-V_{m+1}^{(1)}(t). 
\end{align*}
Similarly, we have
\begin{align*}
&\frac{d}{dt}V_{m}^{(2)}(t)
\\
&=
(-1)^{m}
\left[\sum_{k=1}^{m-1}
\begin{pmatrix}
m-1\\k
\end{pmatrix}\frac{(-tA)^{k-1}}{(k-1)!}(-A)e^{-tA}u_1
+
\sum_{k=0}^{m-1}
\begin{pmatrix}
m-1\\ k
\end{pmatrix}\frac{(-tA)^{k}}{k!}(-A)e^{-tA}u_1
\right]
\\
&=
(-1)^{m}
\left[\sum_{k=1}^{m-1}
\begin{pmatrix}
m-1\\k
\end{pmatrix}\frac{(-tA)^{k-1}}{(k-1)!}(-A)e^{-tA}u_1
+
\sum_{k=1}^{m}
\begin{pmatrix}
m-1\\k-1
\end{pmatrix}\frac{(-tA)^{k-1}}{(k-1)!}(-A)e^{-tA}u_1
\right]
\\
&=
(-1)^{m}
\sum_{k=1}^{m}
\begin{pmatrix}
m\\k
\end{pmatrix}\frac{(-tA)^{k-1}}{(k-1)!}(-A)e^{-tA}u_1
\end{align*}
and then 
\begin{align*}
\int_0^te^{-(t-s)A}\Big(\frac{d}{dt}V_{m}^{(2)}(s)\Big)\,ds
&=
(-1)^{m}
\sum_{k=1}^{m}
\begin{pmatrix}
m\\k
\end{pmatrix}
\frac{(-tA)^{k}}{k!}e^{-tA}u_1
\\
&=-V_{m+1}^{(1)}(t)+(-1)^{m+1}e^{-tA}u_1. 
\end{align*}
Combining the above two identities, we obtain the assertions for $m\in\N$. 
\end{proof}

We prepare the formula of higher order derivative of the functions $V_m$. 
\begin{lemma}\label{lem:V-derivative}
Let $V_m$ be as in Definition \ref{def:V_m}.
For every $m, \ell \in\N$, one has
\[
\frac{d^\ell}{dt^\ell}V_m(t)
=
(-A)^\ell V_{m,\ell}(t), 
\]
where
\[V_{m,\ell}(t)
=
(-1)^m\left[
\sum_{j=0}^m
\begin{pmatrix}
\ell+m-1\\\ell+j-1
\end{pmatrix}\frac{(-tA)^j}{j!}e^{-tA}v_0
+
\sum_{k=0}^{m-1}
\begin{pmatrix}
\ell+m-1\\\ell+k
\end{pmatrix}\frac{(-tA)^k}{k!}e^{-tA}u_1
\right].
\]
\end{lemma}
\begin{proof}
In Lemma \ref{lem:asm:succ}, the following inequalities are already proved:
\begin{align*}
\frac{d}{dt}V_{m}^{(1)}(t)
&=
(-1)^{m}
(-A)\left(\sum_{j=0}^{m}
\begin{pmatrix}
1+(m-1)\\1+(j-1)
\end{pmatrix}\frac{(-tA)^{j}}{j!}e^{-tA}v_0\right),
\\
\frac{d}{dt}V_{m}^{(2)}(t)
&=
(-1)^{m}
(-A)\left(\sum_{k=0}^{m-1}
\begin{pmatrix}
1+(m-1)\\1+k
\end{pmatrix}\frac{(-tA)^{k}}{k!}e^{-tA}u_1\right).
\end{align*}
As the same manner as in these proofs, we can directly check the desired assertion.  
\end{proof}

The following lemma gives integrability and boundness of functions $V_m$ 
with respect to $t\in(0,\infty)$. 
\begin{lemma}\label{lem:max-reg.V_m}
Let $(u_0,u_1)\in D(A^{\ell+1/2})\times D(A^{\ell+1/2})$. 
Then for every $m,\ell\in\N\cup\{0\}$ 
\begin{gather}
\label{eq:bdd.V_m}
(1+t)^{2\ell}\left\|\frac{d^{\ell}}{dt^{\ell}}V_{m}(t)\right\|^2
\in L^\infty((0,\infty)),
\\
\label{eq:max-reg.V_m}
(1+t)^{2\ell+1}\left\|\frac{d^{\ell+1}}{dt^{\ell+1}}V_{m}(t)\right\|^2
+(1+t)^{2\ell}\left\|A^{\ell+1/2}V_{m,\ell+1}(t)\right\|^2\in L^1((\infty)).
\end{gather}
\end{lemma}
\begin{proof}
By Lemma \ref{lem:V-derivative} and the boundedness of operators 
$(tA)^ne^{-\frac{t}{2}A}$ for $n\in\N$, 
we have
\begin{align*}
(1+t)^{2\ell+1}\left\|\Big(\frac{d}{dt}\Big)^{\ell+1}V_{m}(t)\right\|^2
&\leq 
C(1+t)^{2\ell+1}
\left(
\left\|A^{\ell+1}e^{-\frac{t}{2}A}v_0\right\|^2
+
\left\|A^{\ell+1}e^{-\frac{t}{2}A}u_1\right\|^2
\right),
\\
(1+t)^{2\ell}\left\|A^{\ell+1/2}V_{m,\ell+1}(t)\right\|^2
&\leq 
C(1+t)^{2\ell}
\left(
\left\|A^{\ell+1/2}e^{-\frac{t}{2}A}v_0\right\|^2
+
\left\|A^{\ell+1/2}e^{-\frac{t}{2}A}u_1\right\|^2
\right).
\end{align*}
Applying Lemma \ref{lem:max-reg} with $n=2\ell+1$, 
we obtain the desired estimate.
\end{proof}

\subsection{Proof of higher order asymptotic expansion}

Here we provide the relation between the solution $u$ of \eqref{ADW} 
and the family $V_m$.
\begin{lemma}\label{lem:orderm}
Assume $(u_0,u_1)\in D(A^{1/2})\times D(A^{1/2})$. 
Let $\{U_m\}_{m\in\N}$ be a weak solution of 
\begin{align}\label{eq:aux-U_m}
\begin{cases}
U_{m+1}''+AU_{m+1}+U_{m+1}'=-V_{m}', \quad t>0,
\\
(U_{m+1},U_{m+1}')(0)=(0,(-1)^{m+1}u_1).
\end{cases}
\end{align}
Then $U_{m}=V_m+U_{m+1}'$. 
\end{lemma}
\begin{proof}
The strategy of the proof is essentially the same as Lemma \ref{lem:order1}.
By taking a suitable approximation of $(u_0, u_1)$, 
we can assume that $U_m$ is smooth enough without loss of generality.
Putting $w=V_m+U_{m+1}'$, we see from Lemma \ref{lem:asm:succ} that 
\[
w(0)=V_m(0)+U_{m+1}'(0)=(-1)^mu_1+(-1)^{m+1}u_1=0
\]
and 
\begin{align*}
w'(t)
&=V_m'(t)+U_{m+1}''(t)
\\
&=-AU_{m+1}(t)-U_{m+1}'(t).
\end{align*}
This gives 
\[
\lim_{t\to 0}w'(t)=-AU_{m+1}(0)-U_{m+1}'(0)=(-1)^mu_1
\]
and also
\begin{align*}
w''(t)
&=-AU_{m+1}'(t)+(V_m(t)-w(t))'
\\
&=-A(w(t)-V_m(t))+V_m'(t)-w'(t)
\\
&=-Aw(t)-w'(t)+(AV_m(t)+V_m'(t))
\\
&=-Aw(t)-w'(t)-V_{m-1}'(t),
\end{align*}
where we have used Lemma \ref{lem:asm:succ} at the last line. 
Noting the uniqueness of solution to \eqref{eq:aux-U_m}, we obtain $w=U_m$. 
\end{proof}
Formally speaking, from the viewpoint of 
Lemmas \ref{lem:order1} and \ref{lem:orderm} 
we can see that the solution $u$ of \eqref{ADW} can be decomposed as 
\begin{align*}
u(t)
&=V_0(t)+u(t)-V_0(t)
\\
&=V_0(t)+U_1'(t)
\\
&=V_0(t)+V_1(t)' + (V_1(t)-U_1(t))'
\\
&=V_0(t)+V_1'(t) + U_2''(t)
\\
&\,\,\,\vdots
\\
&=\sum_{\ell=0}^m \frac{d^\ell}{dt^\ell}V_{\ell}(t)+\frac{d^{m+1}}{dt^{m+1}}U_{m+1}(t).
\end{align*}
By virtue of \eqref{eq:bdd.V_m} in Lemma \ref{lem:max-reg.V_m}, we already have 
for every $\ell=0,\ldots, m$, 
\[
\left\|\frac{d^\ell}{dt^\ell}V_{\ell}(t)\right\|\leq C(1+t)^{-\ell}.
\]
To ensure that the above description 
is the asymptotic expansion, 
we finally prove the following proposition. 
\begin{proposition}\label{prop:U_mm}
Assume that $(u_0,u_1)\in D(A^{m+1/2})\times D(A^{m+1/2})$. Then 
\[
\left\|\frac{d^{m+1}}{dt^{m+1}}U_{m+1}(t)\right\|\leq (1+t)^{-m-\frac{1}{2}}.
\]
\end{proposition}
The first step of the proof of Proposition \ref{prop:U_mm} 
is to show the following lemma. 
\begin{lemma}\label{lem:en-U_ell}
If $(u_0,u_1)\in D(A^{1/2})\times D(A^{1/2})$, then 
\[
(1+t)E(U_{m+1};t)+\|U_{m+1}(t)\|^2\in L^\infty, 
\quad 
(1+t)\|U_{m+1}'(t)\|^2+\|A^{1/2}U_{m+1}(t)\|^2\in L^1
\]
\end{lemma}
\begin{proof}
The strategy is similar to Lemma \ref{lem:en-U_1}.
In view of Lemma \ref{lem:orderm}, we compute
\begin{align*}
&\frac{d}{dt}
\Big[(6+t)E(U_{m+1},t)+2E_*(U_{m+1},t)\Big]
\\
&
=\|U_{m+1}'\|^2+\|A^{1/2}U_{m+1}\|^2
+2(6+t)(U_{m+1}',-U_{m+1}'-V_{m}')
\\
&\quad +4\|U_{m+1}'\|^2+4(U_{m+1},-AU_{m+1}-V_{m}')
\\
&
=-(7+2t)\|U_{m+1}'\|^2-3\|A^{1/2}U_{m+1}\|^2
-2(6+t)(U_{m+1}',V_{m}')
+4(U_{m+1},AV_{m,1})
\\
&
=-(1+t)\|U_{m+1}'\|^2-\|A^{1/2}U_{m+1}\|^2
-(6+t)\|V_{m}'\|^2
+2\|A^{1/2}V_{m,1}\|^2.
\end{align*}
By applying Lemma \ref{lem:max-reg.V_m} \eqref{eq:max-reg.V_m} with $\ell=0$, 
the proof is complete. 
\end{proof}

To give extra decay property of $\frac{d^\ell}{dt^\ell}U_{m+1}$ via induction, 
we use the following lemma.
\begin{lemma}\label{lem:en-induction}
Assume $(\tilde{u}_0,\tilde{u}_1)\in D(A^{1/2})\times D(A^{1/2})$ and $F\in C([0,\infty);D(A))$. 
Let $w$ be the solution of 
\begin{align}\label{eq:aux-w}
\begin{cases}
w''+Aw+w'=AF, \quad t>0,
\\
(w,w')(0)=(\tilde{u}_0,\tilde{u}_1).
\end{cases}
\end{align}
If $w$ and $F$ satisfy additional condition
\[
(1+t)^{2\ell+1}\|AF(t)\|^2+(1+t)^{2\ell}\|A^{1/2}F(t)\|^2+(1+t)^{2\ell-1}
\|w(t)\|^2\in L^1((0,\infty)), 
\]
then 
\[
(1+t)^{2\ell+1}E(w,t)+(1+t)^{2\ell}\|w\|^2\in L^\infty((0,\infty)), 
\quad
(1+t)^{2\ell+1}\|w'(t)\|^2\in L^1((0,\infty)).
\]
\end{lemma}
\begin{proof}
Let $t_0\geq 1$ be a constant determined later. 
Since we can use an approximation
$w_{\ep}(t)=J_\ep w(t)$ with $J_\ep=(I+\ep A)^{-1}$, 
which satisfies
\begin{align}\label{eq:aux-w-ep}
\begin{cases}
w_\ep''+Aw_\ep+w_\ep'=J_\ep AF, \quad t>0,
\\
(w,w')(0)=(J_\ep \tilde{u}_0,J_\ep \tilde{u}_1)\in D(A^{3/2})\times D(A^{3/2}),
\end{cases}
\end{align}
we may assume $w\in C^2([0,\infty);D(A^{1/2}))\cap C^1([0,\infty);D(A))\cap C([0,\infty);D(A^{3/2}))$. 
By direct computation with the use of \eqref{eq:aux-w}, we have
\begin{align*}
&\frac{d}{dt}
\Big[(t_0+t)^{2\ell+1}E(w,t)+(\ell+2)(t_0+t)^{2\ell}E_*(w,t)\Big]
\\
&=
(2\ell+1)(t_0+t)^{2\ell}E(w,t)
+2(t_0+t)^{2\ell+1}(w',-w'+AF)
\\
&\quad
+
2(\ell+2)
\Big[
\ell(t_0+t)^{2\ell-1}E_*(w,t)+(t_0+t)^{2\ell}\|w'\|^2
+(t_0+t)^{2\ell}(w,-Aw-AF)
\Big]
\\
&\leq 
(t_0+t)^{2\ell}
\Big[
(4\ell+5+\tfrac{2\ell(\ell+1)}{t_0+t})\|w'\|^2
-2\|A^{1/2}w\|^2
-
(t_0+t)\|w'\|^2
\Big]
\\
&\quad
+(t_0+t)^{2\ell+1}\|AF\|^2
+(\ell+2)^2(t_0+t)^{2\ell}\|A^{1/2}F\|^2
+4\ell(\ell+1)(t_0+t)^{2\ell-1}\|w\|^2.
\end{align*}
Taking $t_0=6+4\ell+2\ell(\ell+1)$, we obtain the desired estimate.
\end{proof}

Here we prove Proposition \ref{prop:U_mm}. 
\begin{proof}[Proof of Proposition \ref{prop:U_mm}]
By Lemma \ref{lem:en-U_ell}, we have
\[
(1+t)\|U_{m+1}'(t)\|^2\in L^1((0,\infty)).
\]
Here $w=U_{m+1}'$ satisfies
\begin{align}\label{eq:aux-U_m'}
\begin{cases}
w''+Aw+w'=-V_{m}''=-A^2V_{m,2}, \quad t>0,
\\
(w,w')(0)=((-1)^{m+1}u_1,-AU_{m+1}(0)-V_{m}'(0))=((-1)^{m+1}u_1,AV_{m,1}(0))
\end{cases}
\end{align}
in view of the compatibility condition. 
By virtue of Lemma \ref{lem:max-reg.V_m} with $\ell=1$, 
applying Lemma \ref{lem:en-induction} with $\ell=1$, we have
\[
(1+t)^3\|U_{m+1}''(t)\|^2=(1+t)^3\|w'(t)\|^2\in L^1((0,\infty)).
\]
Successively, we apply Lemma \ref{lem:en-induction} 
with $w=\frac{d^k}{dt^k}U_{m+1}$ and $\ell=k+1$ ($k=1,\ldots,m$). 
Then finally, from Lemma \ref{lem:max-reg.V_m} we deduce
\[
(1+t)^{2m+1}
\left\|\frac{d^{m+1}}{dt^{m+1}}U_{m+1}(t)\right\|^2
\leq 
C(1+t)^{2m+1}E\Big(\frac{d^m}{dt^m}U_{m+1},t\Big)\in L^\infty((0,\infty));
\]
note that we need the regularity $(u_0,u_1)\in D(A^{m+1/2})\times D(A^{m+1/2})$
to verify all deductions. 
The proof is complete.
\end{proof}

To close the paper, we give a proof of Theorem \ref{thm:main}. 
\begin{proof}[Proof of Theorem \ref{thm:main}]
Observing the decomposition (at the beginning of this section)
and Lemma \ref{lem:V-derivative}, we have
\[
u(t)=\sum_{\ell=0}^m (-A)^\ell V_{\ell,\ell}(t)+\frac{d^{m+1}}{dt^{m+1}}U_{m+1}(t).
\]
Combining the definition of $V_{\ell,\ell}$ and Proposition \ref{prop:U_mm}, 
we obtain 
the desired assertion.
\end{proof}

\subsection*{Acknowedgements}
This work is partially supported 
by JSPS KAKENHI Grant Number JP18K134450.


\end{document}